\documentclass{amsart}
\usepackage{comment}
\usepackage{amsfonts}
\usepackage{amssymb}
\usepackage{amsthm}
\usepackage{mathtools}   % loads »amsmath«
\usepackage[alphabetic]{amsrefs}
\usepackage{csquotes}
\usepackage{footmisc}
\usepackage{cite}
\usepackage{xpatch}
\xpatchcmd{\bibsection}{*}{}{}{} %% Makes \bibliography show up in TOC
\usepackage{longtable}
\usepackage[all,cmtip]{xy}
\usepackage{fancyhdr}
\usepackage[normalem]{ulem}

\usepackage{graphicx}
\usepackage{hyperref}

\usepackage{tikz-cd}
\usepackage{tikz}
\usepackage{pifont}
\usepackage{aurical}
\usepackage{stmaryrd}
\usepackage[T1]{fontenc}
\usepackage[outline]{contour}% http://ctan.org/pkg/contour
\usepackage{xcolor}% http://ctan.org/pkg/xcolor
\usepackage{etoolbox}
\usepackage{array}
\usepackage{lipsum}

\usepackage{enumitem}

\usepackage{parskip}

\setlist[enumerate,1]{label={(\alph*)}}

\makeatletter
\DeclareFontEncoding{LS2}{}{\@noaccents}
\makeatother
\DeclareFontSubstitution{LS2}{stix}{m}{n}
\DeclareSymbolFont{largesymbolsstix}{LS2}{stixex}{m}{n}
\DeclareMathDelimiter{\lbrbrak}{\mathopen}{largesymbolsstix}{"EE}{largesymbolsstix}{"14}
\DeclareMathDelimiter{\rbrbrak}{\mathclose}{largesymbolsstix}{"EF}{largesymbolsstix}{"15}

\newcommand{\genlegendre}[4]{%
  \genfrac{(}{)}{}{#1}{#3}{#4}%
  \if\relax\detokenize{#2}\relax\else_{\!#2}\fi
}

\newcommand{\WIP}[1]{\scalebox{#1}{\begin{tikzpicture}[limb/.style={line cap=round,line width=1.5mm,line join=bevel}]
\draw[line width=2mm,rounded corners,fill=yellow] (-2,0) -- (0,-2) -- (2,0) -- (0,2) -- cycle;
\fill (1.5mm,7mm) circle (1.5mm);
\fill(0,-7.5mm) -- ++(10mm,0mm) -- ++(120:2mm)--++(100:1mm)--++(150:2mm) arc (70:170:2.5mm and 1mm);
\draw[limb] (-7.5mm,-6.5mm)--++(70:4mm)--++(85:4mm) coordinate(a)--++(-45:5mm)--(-2.5mm,-6.5mm);
\fill[rotate around={45:(a)}] ([shift={(-0.5mm,0.55mm)}]a) --++(0mm,-3mm)--++
        (7mm,-0.5mm)coordinate(b)--++(0mm,4mm)coordinate(c)--cycle;
\draw[limb] ([shift={(-0.6mm,-0.4mm)}]b) --++(-120:5mm) ([shift={(-0.5mm,-0.5mm)}]c) --++
        (-3mm,0mm)--++(-100:3mm)coordinate (d);
\draw[ultra thick] (d) -- ++(-45:1.25cm);
\end{tikzpicture}}}

\robustify{\WIP}

\definecolor{cec1d24}{RGB}{236,29,36}
\definecolor{cffffff}{RGB}{255,255,255}

\let\int\relax\DeclareMathOperator*{\int}{int}
\newcommand{\lra}{\longrightarrow}
\newcommand{\ra}{\rightarrow}

\newcommand{\bA}{\mathbb{A}}

\newcommand{\bC}{\mathbb{C}}

\newcommand{\bG}{\mathbb{G}}

\newcommand{\bP}{\mathbb{P}}
\newcommand{\bQ}{\mathbb{Q}}

\newcommand{\bZ}{\mathbb{Z}}

\newcommand{\cE}{\mathcal{E}}

\newcommand{\cM}{\mathcal{M}}

\newcommand{\cO}{\mathcal{O}}

\newcommand{\cX}{\mathcal{X}}

\newcommand{\fc}{\mathfrak{c}}

\newcommand{\fh}{\mathfrak{h}}

\newcommand{\fD}{\mathfrak{D}}

\newcommand{\fF}{\mathfrak{F}}

\newcommand{\Aut}{\operatorname{Aut}}

\newcommand{\Out}{\operatorname{Out}}

\newcommand{\Gal}{{\operatorname{Gal}}}

\newcommand{\Hom}{\operatorname{Hom}}

\newcommand{\Conf}{\operatorname{Conf}}

\newcommand{\SL}{\operatorname{SL}}
\newcommand{\PSL}{\operatorname{PSL}}
\newcommand{\PGL}{\operatorname{PGL}}

\newcommand{\rightiso}{\stackrel{\sim}{\longrightarrow}}

\newcommand{\Pic}{\operatorname{Pic}}

\newcommand{\Spec}{\operatorname{Spec}}

\newcommand{\pr}{\operatorname{pr}}

\newcommand{\Aff}{\operatorname{Aff}}

\theoremstyle{definition}\newtheorem{defn}{Definition}[section]
\theoremstyle{remark}
\theoremstyle{remark}
\theoremstyle{remark}\newtheorem{remark}[defn]{Remark}
\theoremstyle{plain}
\theoremstyle{remark}\newtheorem*{remark*}{Remark}
\theoremstyle{remark}
\theoremstyle{remark}
\theoremstyle{definition}
\theoremstyle{remark}
\theoremstyle{def inition}

\theoremstyle{plain}
\theoremstyle{plain}
\theoremstyle{plain}\newtheorem{prop}[defn]{Proposition}
\theoremstyle{plain}\newtheorem{thm}[defn]{Theorem}
\theoremstyle{plain}
\theoremstyle{plain}\newtheorem{lemma}[defn]{Lemma}
\theoremstyle{plain}\newtheorem{cor}[defn]{Corollary}
\theoremstyle{plain}
\theoremstyle{plain}\newtheorem*{thm*}{Theorem}
\theoremstyle{plain}\newtheorem*{conj*}{Conjecture}
\theoremstyle{plain}\newtheorem*{prop*}{Proposition}

\theoremstyle{plain}

\theoremstyle{plain}
\newtheorem{maintheorem}{Theorem}

\newtheorem{maincor}[maintheorem]{Corollary}

\newcommand{\para}[1]{\noindent \textbf{#1.}}

\usepackage[normalem]{ulem}

\title[Families of elliptic curves over the four-pointed configuration space]{Families of elliptic curves over the four-pointed configuration space and exceptional sequences for the braid group on four strands}

\newcommand{\mbf}[1]{\mathbf{#1}}
\newcommand{\bfD}{\mbf{D}}
\newcommand{\bfE}{\mbf{E}}
\newcommand{\bfP}{\mbf{P}}
\newcommand{\bfQ}{\mbf{Q}}

\newcommand{\wt}[1]{\widetilde{#1}}

\renewcommand{\int}{\operatorname{int}}

\author{William Y. Chen}
\author{Nick Salter}

\address{William Y. Chen: Department of Mathematics, Rutgers University, Hill Center for the Mathematical Sciences, 110 Frelinghuysen Rd, Piscataway, NJ 08854-8019}
\email{oxeimon@gmail.com}

\address{Nick Salter: Department of Mathematics, University of Notre Dame, 255 Hurley Building, Notre Dame, IN 46556}
\email{nsalter@nd.edu}
\date{February 16, 2024}

\begin{document}

\maketitle

\begin{abstract}
    We show that the configuration space of four unordered points in $\mathbb{C}$ with barycenter 0 is isomorphic to the space of triples $(E,Q,\omega)$, where $E$ is an elliptic curve, $Q\in E^\circ$ a nonzero point, and $\omega$ a nonzero holomorphic differential on $E$. At the level of fundamental groups, our construction unifies two classical exceptional exact sequences involving the braid group $B_4$: namely, the sequence $1\rightarrow F_2\rightarrow B_4\rightarrow B_3\rightarrow 1$, where $F_2$ is a free group of rank 2, related to Ferrari's solution of the quartic, and the sequence $1\rightarrow \mathbb{Z} \rightarrow B_4\rightarrow\operatorname{Aut}^+(F_2)\rightarrow 1$ of Dyer-Formanek-Grossman\cite{DFG82}.
\end{abstract}

\section{Introduction}

Let $\Conf_n = \Conf_n(\bC)$ denote the configuration space of ordered $n$-tuples of distinct points in $\bC$, and let $\Conf_{[n]} = \Conf_{[n]}(\bC)$ denote the quotient under the free action of $S_n$ by permuting coordinates. I.e. $\Conf_{[n]}$ is the space of configurations of $n$ unordered distinct points, or equivalently the space of monic squarefree polynomials of degree $n$. It is a classical fact that for squarefree $f(X)$ of degree $n$, the equation $Y^2 = f(X)$ defines a hyperelliptic curve of genus $\lfloor \frac{n-1}{2} \rfloor$.

For $n = 3,4$, such curves have genus $1$, and so become elliptic curves when endowed with a choice of basepoint. For $f(X) = X^3 + a_2X^2 + a_4X + a_6$ of degree $n = 3$, the curve $E_f$ given by $Y^2 = f(X)$ is totally ramified over $\infty$, and so $\infty$ provides a consistent choice of distinguished point for this family of genus $1$ curves over $\Conf_{[3]}$. Moreover, the differential $\omega_f := -\frac{dX}{2Y}$ is holomorphic and nonzero, and hence the map $f\mapsto (E_f,\omega_f)$ defines a $\bG_a$-torsor $\Conf_{[3]}\ra\lambda^\circ$, where $\bG_a\cong\bC$ is the group of translations, and $\lambda^\circ$ is the complement of the zero section in the Hodge line bundle on the moduli stack of elliptic curves $\cM_{1,1}$ ($\lambda^\circ$ is an algebraic variety whose points correspond to isomorphism classes of pairs $(E,\omega)$, where $\omega$ is a nonzero holomorphic differential on $E$). This map induces the well-known sequence of groups
$$B_3\cong\pi_1(\Conf_{[3]})\cong\pi_1(\lambda^\circ)\lra\pi_1(\cM_{1,1})\cong\SL_2(\bZ)\twoheadrightarrow\PSL_2(\bZ)$$
which identifies the braid group $B_3$ with the universal central extension of the modular group $\PSL_2(\bZ)$. These maps moreover make sense integrally over $\Spec\bZ[1/2]$ \cite[\S2.2]{KM85}. The composite $\Conf_{[3]}\ra\lambda^\circ\ra\cM_{1,1}$ was also studied in recent work of Huxford--Schillewaert \cite{huxford}, where it was shown to play a key role in the classification of holomorphic maps on $\Conf_{[3]}$.

The purpose of this paper is describe an analogous geometric construction in the case $n = 4$, which refines the above situation for $n = 3$. Let $\cE^\circ$ denote the universal elliptic curve over $\cM_{1,1}$. The fiber product $\cE^\circ\times_{\cM_{1,1}}\lambda^\circ$ is an algebraic variety whose points correspond to triples $(E,Q,\omega)$, where $E$ is an elliptic curve with origin $O$, $Q\in E - O$ is a nonzero point, and $\omega$ is a nonzero holomorphic differential on $E$. Let $\Aff$ denote the stabilizer of $\infty\in\bP^1$ under the action of $\PGL_2(\bC)$, and let $\bG_a\cong\bC$ denote the subgroup of translations.

\begin{maintheorem}\label{thm_A} There is an isomorphism of algebraic varieties
\[\alpha : \Conf_{[4]}/\bG_a\rightiso\cE^\circ\times_{\cM_{1,1}}\lambda^\circ\]
Moreover, the isomorphism $\alpha$ extends to an isomorphism of varieties over $\bZ[1/2]$.
\end{maintheorem}
At the level of sets, the isomorphism $\alpha$ can be described as follows. Given $\tau = \{a,b,c,d\}\in\Conf_{[4]}$, there is a unique double cover $\pi : E_\tau\ra\bP^1$ branched only above $\{a,b,c,d\}$. We make $E_\tau$ into an elliptic curve by taking the origin $O$ to be \emph{any} lift $\wt{\infty}$ of $\infty\in\bP^1$. The preimages $\tilde{a},\tilde{b},\tilde{c},\tilde{d}\in E_\tau$ form a nontrivial coset under the subgroup $E_\tau[2]$. Let $Q_\tau$ denote their common image under the doubling map $[2] : E_\tau\ra E_\tau$.  Finally, if $z$ denotes the coordinate on $\bP^1$, then we let $\omega_\tau$ be the holomorphic differential whose restriction to $O = \wt{\infty}$ corresponds to the differential $d\frac{1}{z}$ at $\infty\in\bP^1$. We will see in Lemma \ref{lemma_Q} that $Q_\tau$ is just the other point above $\infty$ (distinct from $\wt{\infty}$). In particular, the isomorphism class of $(E_\tau,Q_\tau,\omega_\tau)$ is independent of the choice of $\wt{\infty}$.

To go the other way, given $(E,Q,\omega)$, let $Q/2$ denote the preimage of $Q$ under the doubling map $[2] : E\ra E$; thus $Q/2$ is a nontrivial coset of $E[2]$. Let $\iota_{Q/2}$ denote the unique involution of $E$ fixing $Q/2$ pointwise.\footnote{This is the negation map relative to the group structure on $E$ obtained by choosing any point in $Q/2$ as the origin.} Then by Lemma \ref{lemma_Q}, $\{O,Q\}$ is a free orbit of $\iota_{Q/2}$, and $E/\iota_{Q/2}$ is a curve of genus 0. The set of isomorphisms $E/\iota_{Q/2}\rightiso\bP^1$ sending the orbit $\{O,Q\}$ to $\infty\in\bP^1$ is an $\Aff$-torsor. Since $\Aff/\bG_a = \bG_m$ acts on the cotangent space at $\infty$ by scaling, the subset of isomorphisms which identify $\omega|_O$ with $(d\frac{1}{z})|_\infty$ is a torsor under the subgroup $\bG_a\subset\Aff$.\footnote{Let $u := \frac{1}{z}$. The translation $z\mapsto z+a$ sends $u$ to $\frac{u}{1+au}$ with differential $\frac{du}{1+au}$, which agrees with $du$ at $z = \infty$.} The corresponding $\bG_a$-orbit of images of $Q/2$ in $\bP^1$ is then a point of $\Conf_{[4]}/\bG_a$. One easily checks that the two constructions above are mutually inverse.

The characteristic 0 part of Theorem \ref{thm_A} is proven in \S\ref{ss_mapping_Conf4}. In \S\ref{ss_arithmetic} we describe how to extend $\alpha$ to an isomorphism of schemes over $\bZ[1/2]$, using the moduli-theoretic models of $\Conf_{[4]}$ and $\cE^\circ\times_{\cM_{1,1}}\lambda^\circ$. In \S\ref{ss_equations}, we describe explicit equations for the universal elliptic object over $\Conf_{[4]}$ corresponding to $\alpha$.

Theorem \ref{thm_A} can be specialized to yield two classical exact sequences involving $B_4$, both special to the case $n = 4$. The first of these originates with Ferrari's 1540 solution of the quartic via the method of {\em resolvent cubics}. In modern language, this associates the quartic with roots $\{a,b,c,d\}$ to the cubic with roots $\{ab+cd, ac+bd, ad+bc\}$. This map has the remarkable property that if the original roots are distinct, so too are the roots of the resolvent, so that Ferrari's map restricts to
\[
\fF: \Conf_{[4]} \to \Conf_{[3]}.
\]
See \cite{huxford} for further details. Let $\sigma_1,\ldots,\sigma_{n-1}$ be the standard generators of the braid group $B_n$.  The induced map on fundamental groups $\fF_* : B_4\ra B_3$ sends $\sigma_1,\sigma_2,\sigma_3$ to $\sigma_1,\sigma_2,\sigma_1$ respectively, with kernel $F$ freely generated by $\sigma_1\sigma_3^{-1}$ and its conjugate $\sigma_2\sigma_1\sigma_3^{-1}\sigma_2^{-1}$. Thus we have a sequence
\begin{equation}\label{eq_1}
    1\lra F\lra B_4\lra B_3\lra 1.
\end{equation}
We remark that this sequence provides a lift of the exceptional surjection $S_4 \to S_3$ of symmetric groups to the level of braid groups. 

In \cite{DFG82}, Dyer, Formanek, and Grossman showed, via explicit calculations in $B_4$, that the conjugation action of $B_4$ on the subgroup $F$ described above induces a surjection $B_4\ra\Aut^+(F)$ with kernel the center $Z(B_4)$, where $\Aut^+(F)$ denotes the index 2 subgroup of $\Aut(F)$ inducing determinant 1 automorphisms of $H_1(F,\bZ)\cong\bZ^2$. This leads to an exact sequence:
\begin{equation}\label{eq_2}
    1\lra Z(B_4)\lra B_4\lra\Aut^+(F)\lra 1.
\end{equation}
Here, we recall that for $n\ge 2$ the center $Z(B_n)$ is infinite cyclic, generated by $(\sigma_1\cdots\sigma_{n-1})^n$. This sequence was used in \cite{CLT23} to show that the free group of rank 2 admits infinitely many finite simple groups as quotients with \emph{characteristic} kernel, thus providing infinitely many counterexamples to a conjecture of \cite{Chen18} in the theory of noncongruence modular curves. A main motivation for this paper was to give a geometric interpretation to this sequence. Indeed, both sequences \eqref{eq_1},\eqref{eq_2} can be deduced from Theorem \ref{thm_A} as follows:

\begin{maincor}\label{cor_B} Composing the isomorphism $\alpha$ of Theorem \ref{thm_A} with the projection $\cE^\circ\times_{\cM_{1,1}}\lambda^\circ$ to $\lambda^\circ$ induces, on $\pi_1$, the sequence \eqref{eq_1}. Composing $\alpha$ with the projection to $\cE^\circ$, induces, on $\pi_1$, the sequence \eqref{eq_2}.    
\end{maincor}

In the second part of the corollary, we have used the fact that $\pi_1(\cE^\circ)\cong\Aut^+(F)$, where $F$ is understood as the fundamental group of a fiber of the map $\cE^\circ\ra\cM_{1,1}$.

We note that the sequence \eqref{eq_2} can also be understood via Birman--Hilden's theory of the hyperelliptic mapping class group (cf. \cite{birmanhilden71} or \cite[Section 9.4.1]{farbmargalit}) as follows. Birman-Hilden theory gives an isomorphism $B_4/Z(B_4) \cong \text{PMod}(\Sigma_{1,2})$, where $\text{PMod}(\Sigma_{1,2})$ denotes the mapping class group of a surface $\Sigma_{1,2}$ of genus $1$ with $2$ individually-distinguishable marked points. The group $F$ then arises as the ``point-pushing subgroup'', i.e. the kernel of the forgetful map $\text{PMod}(\Sigma_{1,2}) \to \text{Mod}(\Sigma_{1,1})$. By the Teichmuller uniformization of the moduli stacks of marked curves \cite[\S12.1]{FM11}, this forgetful map is induced by the universal family $\cE^\circ\ra\cM_{1,1}$.

\para{Acknowledgements} The authors would like to thank Benson Farb, Pierre Deligne, Aaron Landesman, and Andy Putman for comments on a preliminary draft. The second author is supported by NSF award no. DMS-2153879.

\section{Proof of Theorem A}
For short, let $\cX := \cE^\circ\times_{\cM_{1,1}}\lambda^\circ$. To prove Theorem \ref{thm_A}, we will describe a morphism of algebraic varieties
\[\alpha' : \Conf_{[4]}\lra\cX\]
and verify that it descends to an isomorphism $\alpha : \Conf_{[4]}/\bG_a\rightiso\cX$, where $\bG_a\subset\Aff$ is the subgroup of translations. In the discussion following Theorem \ref{thm_A} we have already described $\alpha,\alpha'$ at the level of sets. To check that they are morphisms of algebraic varieties, instead of choosing coordinates and checking that the map is defined by polynomial equations, we will describe the map in terms of the moduli interpretations of $\Conf_{[4]}$ and $\cX$.

When describing algebro-geometric objects, for simplicity we will by default work over $\bC$. Thus algebraic varieties are complex algebraic varieties, schemes are $\bC$-schemes, etc. However, our methods easily extend to all odd characteristics, see \S\ref{ss_arithmetic}.

\subsection{Configuration spaces}
In this section, we recall the moduli interpretation of $\Conf_n$ and $\Conf_{[n]}$. For $d\ge 1$, let $\bA^d = \Spec\bC[x_1,\ldots,x_d]$ denote affine $d$-space. The space $\Conf_n$ admits a universal family $\Conf_n\times\bA^1\ra\Conf_n$, equipped with $n$ tautological sections $\Sigma_1,\ldots,\Sigma_n$, where $\Sigma_i$ sends $(x_1,\ldots,x_n)$ to $((x_1,\ldots,x_n),x_i)$.

For a scheme $T$, a map $T\ra\Conf_n$ defines, by pullback, $n$ mutually disjoint sections $\sigma_{1},\ldots,\sigma_{n}$ of the relative affine line $\bA^1_T = \bA^1\times T$. Conversely, if $\Delta\subset\bA^n$ denotes the big diagonal consisting of points where at least two coordinates coincide, then for any $n$ mutually disjoint sections $\sigma_1,\ldots,\sigma_n : T\ra\bA^1_T$, their product defines a map $T\ra\Conf_n = \bA^n - \Delta$ which identifies $\sigma_i$ with the pullback of $\Sigma_i$. Thus, we have a canonical bijection
$$\Conf_n(T) := \Hom(T,\Conf_n) \cong \{\text{Mutually disjoint sections $\sigma_1,\ldots,\sigma_n : T\ra\bA^1_T$}\}$$
which is functorial in $T$. When $T = \Spec\bC$, the complex manifold $\bC^n - \Delta$ is thus recovered as $\Conf_n(\Spec\bC)$.

The unordered configuration space $\Conf_{[n]}$ is the quotient of $\Conf_n$ by the free action of $S_n$. In the next proposition we recall the analogous description of $\Conf_{[n]}(T) := \Hom(T,\Conf_{[n]})$. Let $\fD_n(T)$ denote the set of effective divisors $D\subset\bA^1_T$ which are finite \'{e}tale of degree $n$ over $T$.

\begin{prop} For a scheme $T$, there is a bijection $\Conf_{[n]}(T)\rightiso \fD_n(T)$ functorial in $T$.
\end{prop}
\begin{proof} Given a map $T\ra\Conf_{[n]}$, let $P := \Conf_n\times_{\Conf_{[n]}} T$ be the pullback. The projection $P\ra\Conf_n$ defines mutually disjoint sections $\sigma_1,\ldots,\sigma_n$ of $\bA^1_P$. Identifying $\sigma_i$ with its image in $\bA^1_P$, the map $\sqcup_i\sigma_i\subset\bA^1_P\ra\bA^1_T$ is $S_n$-invariant, and hence factors through $(\sqcup_i\sigma_i)/S_n\hookrightarrow\bA^1_T$. This is the desired object of $\fD_n(T)$.

Conversely, for $D\in\fD_n(T)$, let $P$ denote the complement of the big diagonal in the $n$th fibered power $D^n = D\times_T\cdots\times_T D$. The action of $S_n$ on $D^n$ by permuting coordinates makes $P$ into an $S_n$-torsor over $T$, and $D\times_T P$ is a totally split cover of $P$, with canonical sections $\sigma_1,\ldots,\sigma_n : P\ra D\times_T P\subset\bA^1_P$ induced by the projections $\pr_i : P\subset D^n\lra D$. This defines a $S_n$-equivariant map $P\ra\Conf_n$, whence a map $P/S_n = T\ra\Conf_{[n]}$.

One checks that the two constructions defined above are mutually inverse.
\end{proof}

\begin{comment}
Let $\fc : \Conf_n\ra\bC$ be given by taking the ``center of mass'' of a configuration:
\begin{eqnarray*}
    \fc : \Conf_n & \lra & \bC \\
    (x_1,\ldots,x_n) & \mapsto & \frac{1}{n}(x_1 + \cdots + x_n)
\end{eqnarray*}
This map is $S_n$-invariant, and hence descends to a map $\Conf_{[n]}\ra\bC$, which we also denote by $\fc$. Let $\Conf_n^0\subset\Conf_n$ and $\Conf_{[n]}^0\subset\Conf_{[n]}$ denote the (smooth) subvarieties given by $\fc = 0$. For a scheme $T$, $\Conf_{[n]}^0(T)$ is in bijection with the divisors $D\in\fD_n(T)$ such that over some trivializing \'{e}tale cover $T'\ra T$, the sum of the sections in $D_{T'}$ is equal to the zero section of $\bA^1_{T'} = \bG_{a,T'}$.

There is a natural action of $\bG_a$ on $\Conf_{[n]}$, and given $D\in \fD_n(T) = \Conf_{[n]}(T)$, there is a unique $\beta_D\in\bG_a(T)$ such that $\beta\cdot D\in\Conf_{[n]}^0(T)$. Specifically, if we view $D$ as a map $T\ra\Conf_{[4]}$, and view $\fc(D)$ as a function $T\ra\Conf_{[n]}\ra\bC\cong\bG_a$, then $\beta_D = -\fc(D)$. It follows that

\begin{prop} The map $D\mapsto\beta_D\cdot D$ defines a (trivial) $\bG_a$-torsor
$$\Conf_{[4]}\lra\Conf_{[4]}^0$$
In particular, we have an isomorphism $\Conf_{[4]}^0\cong\Conf_{[4]}/\bG_a$.
\end{prop}
\end{comment}

\subsection{The double cover of $\bP^1$ branched over 4 points}

Recall that an elliptic curve over a scheme $T$ is a proper smooth morphism $E\ra T$ equipped with a section $O : T\ra E$. There is a unique group scheme structure on $E/T$ relative to which $O$ is the zero section. When $T$ is smooth and of finite type, such an object can be viewed as a family of genus 1 Riemann surfaces equipped with a holomorphic family of marked points, parametrized by the complex manifold $T(\bC)$.

For short, write $\bfP := \bP^1_{\Conf_{[4]}}$, and let $\bfD\subset\bfP$ be the universal degree 4 divisor, disjoint from $\infty\subset\bfP$. Let $\tau = \{\tau_1,\tau_2,\tau_3,\tau_4\}\in\Conf_{[4]}$ be a base point, and let $\gamma_1,\ldots,\gamma_4\in\pi_1(\bP^1,\infty)$ be a set of standard ``meridional'' generators winding once counterclockwise around $\tau_1,\ldots,\tau_4$ respectively. Since $\Conf_{[n]}$ has trivial higher homotopy groups \cite{FN62}, the fibration $\bfP - \bfD\ra\Conf_{[4]}$ induces a split exact sequence of homotopy groups
\begin{equation}\label{eq_projective_HES}
   \begin{tikzcd}  1\ar[r] & \pi_1(\bP^1 - \tau,\infty)\ar[r] & \pi_1(\bfP - \bfD, (\tau,\infty))\ar[r] & \pi_1(\Conf_{[4]},\tau)\ar[r]\ar[l,bend right = 20, "\infty_*"'] & 1
   \end{tikzcd}
\end{equation}

\begin{prop} The monodromy representation $B_4\cong\pi_1(\Conf_{[4]},\tau)\ra\Out(\pi_1(\bP^1 - \tau,\infty))$ induced by $\infty_*$ preserves the kernel of the map $\fh : \pi_1(\bP^1 - \tau,\infty)\lra\bZ/2$ sending each $\gamma_i\mapsto 1\mod 2$.
\end{prop}

\begin{proof} The monodromy representation of $B_4$ on $\pi_1(\bP^1 - \tau,\infty)\cong\langle\gamma_1,\gamma_2,\gamma_3,\gamma_4|\gamma_1\gamma_2\gamma_3\gamma_4 = 1\rangle$ is given by the Hurwitz action
\begin{equation}\label{eq_action}
    \sigma_i(\gamma_i) = \gamma_i\gamma_{i+1}\gamma_i^{-1},\quad\sigma_i(\gamma_{i+1}) = \gamma_i,\;\;\text{and}\;\;\sigma_i(\gamma_j) = \gamma_j\text{ for $j\notin\{i,i+1\}$}.
\end{equation}
It is easy to see that the kernel of $\fh$ is preserved by this action.
\end{proof}

As a corollary, we obtain the following extension of the familiar statement that there is a unique double cover of $\bP^1$ branched over four points.
\begin{cor}\label{cor_unique} For a smooth scheme $T$, let $D\in\fD_n(T)$. There is a double cover\footnote{I.e., a finite flat morphism of degree 2.} $\pi : E\ra\bP^1_T$ branched over $D$, \'{e}tale over $\bP^1_T - D$, and split above $\infty\subset\bP^1_T$. This cover is unique up to isomorphism of covers. If $\wt{\infty}\subset E$ is a choice of section lying over $\infty$, then $(E,\infty)$ is an elliptic curve over $T$.
\end{cor}
\begin{proof} The Riemann-Hurwitz formula implies that any such $E$ has genus 1, and hence becomes an elliptic curve when equipped with a choice of section. It remains to show existence and uniqueness. Let $t\in T$ be a geometric point. By \cite[\S XIII, Prop 4.3]{SGA1}, we have a split exact sequence of (\'{e}tale) fundamental groups
\begin{equation}\label{eq_T_HES}
   \begin{tikzcd}  1\ar[r] & \pi_1(\bP^1 - D_t,\infty)\ar[r] & \pi_1(\bP^1_T - D, (t,\infty))\ar[r] & \pi_1(T,t)\ar[r]\ar[l,bend right = 20, "\infty_*"'] & 1
   \end{tikzcd}
\end{equation}
We will view $\pi_1(T,t)$ as embedded inside $\pi_1(\bP^1_T-D,(t,\infty))$. A double cover of $\bP^1_T$ is given by a homomorphism $\rho : \pi_1(\bP_T - D, (t,\infty))\ra\bZ/2$. The cover is split above $\infty$ if and only if $\rho$ restricts to the trivial homomorphism on the right term $\pi_1(T,t)$, and it is branched over $D$ if and only if the restriction to the left term $\pi_1(\bP^1-D_t,\infty)$ is the map $\fh : \pi_1(\bP^1-D,\infty)\ra\bZ/2$ which sends $\gamma_1,\ldots,\gamma_4$ to 1, where $\gamma_1,\ldots,\gamma_4$ denote the standard meridional generators as above. Since the middle term of \eqref{eq_T_HES} is generated by the left and right terms, the desired homomorphism $\rho$ is unique if it exists. To see that it exists, it suffices to note that by the moduli interpretation of $\Conf_{[4]}$, $D\subset\bP^1_T$ is the pullback of $\bfD\subset\bfP$ via a map $T\ra\Conf_{[4]}$, and hence the action of $\pi_1(T,t)$ on $\pi_1(\bP^1-D_t,\infty)$ induced by $\infty_*$ is the pullback of the action \eqref{eq_action}, which preserves the kernel of $\fh$. 
\end{proof}

\begin{remark} The uniqueness part of the corollary is false without the condition that the cover is split above $\infty$. Indeed, the proof shows that there is a bijection between the connected double covers of $\bP^1_T$ branched over $D$ and double covers of $\infty\subset\bP^1_T$. The split double cover is the one described in the corollary. The other double covers give rise to quadratic twists of $E$.
\end{remark}

\begin{remark} The uniqueness part of the corollary is also false without the smoothness assumption on $T$. Indeed, for $T = \Spec\bC[\epsilon]/\epsilon^2$, since deformations of smooth affine schemes are trivial \cite[\S1, Cor 4.8]{HartDT}, the finite \'{e}tale double covers of $\bP^1 - \{a_1,\ldots,a_4\}$ given by $Y^2 = (X-a_1)(X-a_2)(X-a_3)(X-a_4)$ depend only on the residue class of $a_i\mod\epsilon$. On the other hand, the map sending an elliptic curve to its branch locus in $\bP^1$ defines an \'{e}tale map $\cM_{1,1}\ra\cM_{0,4}$, and hence induces an isomorphism on tangent spaces. This implies that the nontrivial deformations of the branch locus $\{a,b,c,d\}$ corresponds to nontrivial deformations of the elliptic curve. To recover uniqueness in this general case, one should additionally assert that \'{e}tale locally near each branch point $a_i$, the cover is given by $\cO_T[x,y]/(y^2-(x-a_i))$ \cite[XIII, Prop 5.5]{SGA1}.
\end{remark}

With notation and hypotheses as in Corollary \ref{cor_unique}, let $O,O'\subset E$ denote the two sections lying above $\infty\subset\bP^1_T$. We view $E$ as an elliptic curve with origin $O$. Let $R\subset E$ be the ramification divisor of $\pi : E\ra\bP^1_T$, viewed as a finite \'{e}tale degree 4 divisor over $T$. Then $R$ is also the fixed point subscheme of an involution on $E$ (generator of $\Gal(\pi)\cong\bZ/2$), and hence is a torsor under the 2-torsion subgroup scheme $E[2]$ acting by translation. Since $R$ does not contain the zero section $O$, $R$ is a nontrivial coset, and hence the multiplication by 2 map $[2] : E\ra E$ sends $R$ to a section $Q\subset E$, disjoint from $O$.
\begin{lemma}\label{lemma_Q} The section $Q := [2]R$ is precisely the other section $O'$ lying over $\infty\subset\bP^1_T$.
\end{lemma}

\begin{remark}\label{rem_well_def} In particular, recalling that $\cE^\circ$ is the category of elliptic curves equipped with a nonzero point, this lemma shows that the objects $((E,O),Q)$ and $((E,Q),O)$ of $\cE^\circ$ are (uniquely) isomorphic, the isomorphism given by the nontrivial element of $\Gal(E/\bP^1_T)$. This also shows that the constructions described in the discussion following Theorem \ref{thm_A} are well-defined.
\end{remark}

\begin{proof}
It suffices to check this on geometric fibers. For a geometric point $t\in T$, let $E_t$ be the fiber, and let $P\in R\subset E_t$ be a ramified point. We will check that $[2]P = O_t'$. Recall that if $E_0$ is a genus 1 curve with a point $o\in E_0$, the group law on the elliptic curve $E_0$ with origin $o$ is defined by pulling back the group structure on $\Pic^0(E_0)$ via the Abel-Jacobi isomorphism $E_0\ra\Pic^0(E_0)$ sending $p\mapsto (p) - (o)$, where parentheses denote the divisor associated to a point \cite[\S III, Prop 3.4]{Sil09}. Since $O_t,O_t'$ are negations of each other with respect to the group structure where $P$ is the origin, we have the identity
$$(O_t) - (P) = -((O_t') - (P))\quad\text{in $\Pic^0(E_t)$}$$
Thus $2(P) = (O_t) + (O_t')$, so $2(P) - 2(O_t) = (O_t') - (O_t)$. This last equality says that $[2]P = O_t'$ relative to the group law with origin $O_t$, as desired.
\end{proof}

\subsection{Mapping $\Conf_{[4]}\ra\cX$}\label{ss_mapping_Conf4}
As a moduli stack, the objects of $\cX := \cE^\circ\times_{\cM_{1,1}}\lambda^\circ$ over a scheme $T$ consist of triples $(E,Q,\omega)$, where $E$ is an elliptic curve over $T$ (with origin $O$), $Q$ is a nowhere vanishing section of $E$, and $\omega$ is a nonzero section of $H^0(E,\Omega^1_{E/T})$. Note that since automorphisms of elliptic curves act faithfully on the set of nonzero differentials, such objects have no nontrivial automorphisms, and hence $\cX$ is an algebraic variety. 

Recall that $\Hom(T,\cX)$ is identified with the set of isomorphism classes of objects of $\cX$ over $T$. Thus to give a map $\alpha' : \Conf_{[4]}\ra\cX$, we must describe an object of $\cX$ over $\Conf_{[4]}$.

Let $\bfD\subset\bfP$ be the universal degree 4 divisor, and let $\pi : \bfE\ra\bfP$ be the unique double cover branched over $\bfD$ and split over $\infty$, as in Corollary \ref{cor_unique}. We fix an (arbitrary) lift $\wt{\infty}\subset\bfE$ of $\infty$. Setting $O := \wt{\infty}$, $(\bfE,O)$ is an elliptic curve over $\Conf_{[4]}$. Let $R\subset\bfE$ be the ramification locus. By the discussion preceding Lemma \ref{lemma_Q}, this is a nontrivial coset under $\bfE[2]$. Let $\bfQ\subset\bfE$ be the image of $R$ under $[2] : \bfE\ra\bfE$. Thus $\bfQ$ is disjoint from $O$, and the preimage above $\infty\subset\bfP$ is $\bfQ\sqcup O$.

Next we define the differential form $\omega_\bfE$. First we recall the algebraic analogue of the classical statement that invariant differential forms on a Lie group are determined by their restriction to the origin.

\begin{prop} Let $f : E\ra T$ be an elliptic curve with zero section $O : T\ra E$. Then there is a canonical isomorphism
\begin{equation}\label{eq_push_pull}
    O^*\Omega^1_{E/T}\cong f_*\Omega^1_{E/T}
\end{equation}
\end{prop}
\begin{proof} Since $E\ra T$ is a smooth group scheme, there is a canonical isomorphism $f^*O^*\Omega^1_{E/T}\cong\Omega^1_{E/T}$ obtained by extending forms defined along $O$ to invariant forms on $E$ \cite[\S4.2, Prop 2]{BLR90}. Applying $f_*$, we have $f_*f^*O^*\Omega^1_{E/T}$. By the projection formula \cite[Prop 5.2.32]{Liu02}, we find
$$f_*\Omega^1_{E/T} \cong f_*f^*O^*\Omega^1_{E/T} = f_*(f^*O^*\Omega^1_{E/T}\otimes\cO_E) = O^*\Omega^1_{E/T}\otimes f_*\cO_E = O^*\Omega^1_{E/T}$$
as desired.
\end{proof}

Let $f : \bfE\ra\Conf_{[4]}$ denote the structure morphism. Since $\pi : \bfE\ra\bfP$ is split above $\infty$, $\pi$ induces an isomorphism $\infty^*\Omega^1_{\bfP/\Conf_{[4]}}\cong\wt{\infty}^*\Omega^1_{\bfE/\Conf_{[4]}}$. Using \eqref{eq_push_pull}, we obtain an isomorphism $\infty^*\Omega^1_{\bfP/\Conf_{[4]}}\cong f_*\Omega^1_{\bfE/\Conf_{[4]}}$. Let $\omega_\bfE\in H^0(\bfE,\Omega^1_{\bfE/\Conf_{[4]}})$ be the differential form corresponding to the global section $d\frac{1}{z}$ of $\infty^*\Omega^1_{\bfP/\Conf_{[4]}}$.

\begin{defn} Let $\alpha' : \Conf_{[4]}\lra\cX := \cE^\circ\times_{\cM_{1,1}}\lambda^\circ$ denote the map defined by the triple $(\bfE,\bfQ,\omega_\bfE)$ described above.
\end{defn}

It follows from Remark \ref{rem_well_def} that $\alpha'$ is independent of the choice of origin $O = \wt{\infty}\subset\bfE$ lying over $\infty\subset\bfP$.

One easily checks that $\alpha'$ is invariant under the action of $\bG_a$, and hence factors through a map $\Conf_{[4]}/\bG_a\ra\cX$. The following theorem completes the proof of the characteristic 0 part of Theorem \ref{thm_A}.

\begin{thm} The map $\alpha : \Conf_{[4]}/\bG_a\ra\cX$ is an isomorphism.
\end{thm}
\begin{proof} Since $\cX$ is normal and $\Conf_{[4]}/\bG_a$ is integral, by Zariski's main theorem \cite[Cor 4.4.6]{Liu02} it suffices to check this at the level of points, but this follows from the discussion following Theorem \ref{thm_A}.
\end{proof}

\subsection{Extension over $\bZ[1/2]$}\label{ss_arithmetic}
Our construction of $\alpha : \Conf_{[4]}/\bG_a\rightiso\cX := \cE^\circ\times_{\cM_{1,1}}\lambda^\circ$ is easily seen to extend to a morphism of stacks over $\bZ[1/2]$, whose definition is exactly the same as in characteristic 0. Our arguments also easily extend, essentially word for word, to show that this extension of $\alpha$ is an isomorphism. Here we make some remarks.

The sequence \eqref{eq_T_HES} is no longer exact if $T$ is not a $\bQ$-scheme. However, the same reference \cite[\S XIII.4]{SGA1} describes how to fix this by taking appropriate quotients of the fundamental groups. Since Corollary \ref{cor_unique} uses this sequence to study double covers, we are okay as long as 2 is invertible on $T$.

In characteristic $p = 2$, one issue is that branched double covers are wildly ramified, and since wild ramification contributes more to arithmetic genus than tame ramification \cite[II, Theorem 5.9]{Sil09}, the cover $E\ra\bP^1_T$ constructed in Corollary \ref{cor_unique} would not be elliptic. Another issue is that the doubling map $[2]$ is not 4-to-1 in characteristic 2.

In characteristic $p = 3$, there is a subtlety that nontrivial automorphisms of elliptic curves $E$ can have nonzero fixed points on $H^0(E,\Omega^1_E)$. This happens only for the supersingular elliptic curve $E : y^2 = x^3-x$ of $j$-invariant $0 = 1728$ and automorphism group the nonabelian semidirect product $\bZ/3\rtimes\bZ/4$ \cite[Exercise A.1]{Sil09}. The kernel of $\Aut(E)$ acting on $H^0(E,\Omega^1_E)$ is the normal subgroup $\bZ/3$, generated by $(y,x)\mapsto (y,x+1)$, which acts on $E$ without nonzero fixed points. Thus, $\lambda^\circ$ is not a scheme above $E$, but $\cE^\circ$, and hence $\cX$, is still a scheme. Indeed, $\Conf_{[4]}/\bG_a$ is a scheme, and (the extension of) $\alpha$ is an isomorphism of schemes.

\subsection{Explicit equations for $(\bfE,\bfQ,\omega_\bfE)$}\label{ss_equations}
If $a,b,c,d$ are the tautological coordinates on $\Conf_4$, then an explicit equation for $\bfE$ is given by the elliptic curve whose affine equation is:
\begin{equation}\label{eq_explicit_quartic_new}
    y^2 = (x-a)(x-b)(x-c)(x-d)
\end{equation}
and the covering map to $\bP^1$ given by $(x,y)\mapsto x$. This equation visibly has the desired branching behavior above $\bfP$. It follows from standard calculations (see, e.g., \cite[\S II.2]{Sil09}), that an affine neighborhood around infinity is given by
\begin{equation}\label{eq_affine_patch}
    w^2 = (1-av)(1-bv)(1-cv)(1-dv)
\end{equation}
where the gluing maps are given by $v = 1/x$, $w = y/x^2$, and the map to $\bfP$ is given by $(v,w)\mapsto 1/v$. It follows that $\{O,O'\}$ correspond, in some order, to $\{(0,1),(0,-1)\}$ in $(v,w)$-coordinates, where by remark Lemma \ref{lemma_Q}, $O' = \bfQ$. In particular, the smooth compactification of \eqref{eq_explicit_quartic_new} is split above $\infty\subset\bfP$. Thus, by the uniqueness part of Corollary \ref{cor_unique}, \eqref{eq_explicit_quartic_new} is indeed an equation for $\bfE$.

The holomorphic differential $\frac{dx}{y}$ on the $(x,y)$-patch becomes $-\frac{dv}{w}$ on the $(v,w)$-patch. Since $d\frac{1}{z}$ on $\bP^1$ pulls back to $dv$, we must have $\omega_\bfE = \pm\frac{dv}{w} = \mp\frac{dx}{y}$, the sign corresponding to the sign of $O = (0, \pm 1)$.

%\subsection{An alternative picture of the map $\Conf_{[4]}\ra\cE^\circ$}
%The map $\Conf_{[4]}\ra\cE^\circ$ can also be understood as follows. Let $\cY(2)$ be the moduli stack of elliptic curves with a trivialization of its points of order 2. There are three such points, and the forgetful map $\cY(2)\ra\cM_{1,1}$ is a $S_3$-torsor. Let $\cE(2)$ denote the universal elliptic curve over $\cY(2)$, with trivialized order 2 points. Let $\cE(2)^* := \cE(2) - \cE(2)[2]$. For $(a,b,c,d)\in\Conf_4$, curve \eqref{eq_explicit_quartic_new} defines a map $\Conf_4\ra\cE(2)^*$, where the ramified point above $d$ is the origin, the ramified points above $a,b,c$ describe the trivialization of the order-2-points, and \emph{any} preimage of $\infty\subset\bfP$ defines the non-2-torsion section.

\section{Proof of Corollary B}
In this section we check that sequences \eqref{eq_1} and \eqref{eq_2} can indeed be obtained from our main isomorphism $\alpha$. Since $\cE^\circ,\cM_{1,1},\lambda^\circ$ are all orbifold $K(\pi,1)$ spaces, we find that $\pi_1(\cE^\circ \times_{\cM_{1,1}} \lambda^\circ)$ is a group-theoretic pullback of the projection maps $\pi_1(\cE^\circ) \to \pi_1^{orb}(\cM_{1,1})$ and $\pi_1(\lambda^\circ) \to \pi_1^{orb}(\cM_{1,1})$.

By Theorem \ref{thm_A}, after identifying the various fundamental groups involved, we obtain the following pullback diagram:
\begin{equation}\label{eq_square}
\begin{tikzcd}
    B_4 \ar[r] \ar[d] & B_3 \ar[d] \\
    \Aut^+(F_2) \ar[r]& \SL_2(\bZ).
\end{tikzcd}
\end{equation}
The long exact sequence in homotopy for the bundles $\cE^\circ$ and $\lambda^\circ$ over $\cM_{1,1}$ yield the following short exact sequences of fundamental groups:
\begin{eqnarray*}
        &1\lra Z(B_3)^2 \lra B_3\lra \SL_2(\bZ)\lra 1 \\
        &1\lra F_2\lra \Aut^+(F_2)\lra \SL_2(\bZ)\lra 1.
\end{eqnarray*}
Here, recall that $Z(B_3)$ is infinite cyclic, with quotient $B_3/Z(B_3)\cong\PSL_2(\bZ)$, so by $Z(B_3)^2$ we mean the unique subgroup of index 2. By properties of the pullback, the commutative square \eqref{eq_square} can be extended to the following diagram, with all rows and columns exact:
\begin{equation}\label{eq_bigtikz}
\begin{tikzcd}
         &                          & 1 \ar[d]                  & 1 \ar[d]          &   \\
         &                          & Z(B_4) \ar[d] \ar[r, "\cong"]  & Z(B_3)^2 \ar[d]        &   \\
1 \ar[r] & F_2 \ar[r] \ar[d, "\cong"] &    B_4 \ar[r] \ar[d]      & B_3 \ar[d] \ar[r] & 1 \\
1 \ar[r] & F_2 \ar[r]               &\Aut^+(F_2) \ar[r]\ar[d]   & \SL_2(\bZ) \ar[r] \ar[d] & 1 \\
         &                          & 1                         &1                  &
\end{tikzcd}
\end{equation}
To prove Corollary \ref{cor_B}, it remains to check that the sequences in \eqref{eq_bigtikz} passing through $B_4$ coming from our main theorem coincide with the sequences \eqref{eq_1} and \eqref{eq_2} of Ferrari and Dyer-Formanek-Grossman, respectively. We shall refer to the two sequences in \eqref{eq_bigtikz} passing through $B_4$ as the ``vertical'' and ``horizontal'' sequences.

The sequence \eqref{eq_2} of Dyer-Formanek-Grossman has kernel the center of $B_4$. Since the vertical sequence comes from a $\bG_m$-torsor, it has central kernel. Since $\Aut^+(F_2)$ has trivial center, it follows that the kernel is the full center, and hence the vertical sequence is isomorphic to the sequence \eqref{eq_2}.

Finally, recall that Ferrari's map $\fF_* : B_4\ra B_3$ is given by sending $\sigma_1,\sigma_2,\sigma_3$ to $\sigma_1,\sigma_2,\sigma_1$ respectively. By \cite[Theorem 2.8(ii)]{huxford}, \emph{every} homomorphism $B_4\ra B_3$ factors through $\fF_*$. Since braid groups are Hopfian, it follows that the surjection $B_4\ra B_3$ in the horizontal sequence is the composition of $\fF_*$ with an automorphism of $B_3$. In particular, the horizontal sequence is isomorphic to Ferrari's sequence \eqref{eq_1}.

%Checking compatibility with \eqref{eq_2} is not so bad: the map $\Conf_{[4]}/\bG_a\cong\cE^\circ\times_{\cM_{1,1}}\lambda^\circ\ra\cE^\circ$ induces a surjection $B_4\ra\Aut^+(F)$ with central kernel. Since $\Aut^+(F)$ has trivial center, the kernel is the full center. A little more detail should be added.

% Checking compatibility with \eqref{eq_1} would probably require drawing some pictures...

% \will{Can you fill out this section?}

\bibliography{references}

\end{document}